\documentclass[11pt]{article}
\usepackage{latexsym}
\usepackage{amsmath,amsthm,amsfonts,amssymb}
\usepackage[T1]{fontenc}
\usepackage[utf8]{inputenc}
\usepackage{aeguill}
\usepackage{url}
\usepackage{enumerate}
\usepackage{mdwlist} 
\usepackage{graphicx}
\usepackage[a4paper,textwidth=15cm,textheight=25cm]{geometry}
\usepackage{xspace}
\usepackage{export}
\usepackage{hyperref}

\newtheorem{thm}{Theorem}[section]

\newtheorem*{thm*}{Theorem}
\newtheorem{dfn}[thm]{Definition} 
\newtheorem*{dfn*}{Definition}

\newtheorem{cor}[thm]{Corollary}
\newtheorem*{cor*}{Corollary}

\newtheorem{prop}[thm]{Proposition} 
\newtheorem*{prop*}{Proposition} 
\newtheorem*{properties*}{Properties} 
 
\newtheorem{lem}[thm]{Lemma} 
\newtheorem*{lem*}{Lemma}

\newtheorem*{claim*}{Claim} 
 
\newtheorem*{fact*}{Fact}

\newtheorem*{qst*}{Question}

\newtheorem*{pb*}{Problem}

\theoremstyle{remark}
 
\newtheorem*{algo*}{Algorithm} 
\newtheorem*{rem*}{Remark}
\newtheorem{rem}[thm]{Remark}
\newtheorem*{example*}{Example}

\newcounter{numEnonceTmpInterne}
\newenvironment{enonce*}[1]{\theoremstyle{plain}\stepcounter{numEnonceTmpInterne}%
\def\a{enoncetmp\alph{numEnonceTmpInterne}}%
\newtheorem*{\a}{#1}\begin{\a}}{\end{\a}}

\makeatletter
\edef\@tempa#1#2{\def#1{\mathaccent\string"\noexpand\accentclass@#2 }}
\@tempa\rond{017}
\makeatother

\newcommand{\es}{\emptyset}
\renewcommand{\phi}{\varphi} 
\newcommand{\m} {{}^{-1}} 
\newcommand{\eps} {\varepsilon}

\newcommand {\ra} {\rightarrow}

\newcommand {\onto} {\twoheadrightarrow}

\newcommand{\ol}[1]{\overline{#1}}

\newcommand {\cala} {{\mathcal {A}}}   
   
\newcommand {\calc} {{\mathcal {C}}}

\newcommand {\calh} {{\mathcal {H}}}   
\newcommand {\cali} {{\mathcal {I}}}

\newcommand {\calp} {{\mathcal {P}}}

\newcommand {\cals} {{\mathcal {S}}}

\newcommand {\calx} {{\mathcal {X}}}

\newcommand {\bbZ} {{\mathbb {Z}}}   

\newcommand{\grp}[1]{\langle #1 \rangle}

\newcommand{\Out} {{\mathrm{Out}}}
\newcommand{\Inn} {{\mathrm{Inn}}}

\newcommand{\Aut} {{\mathrm{Aut}}}

\newcommand{\ad} {\mathrm{ad}}

\newcommand{\Cay} {{\mathrm{Cay}}}

\newcommand{\Inc}{\mathrm{Inc}}

\usepackage{ stmaryrd }

\newcommand{\Z}{{\mathbb {Z}}}

\usepackage{pdfsync}
\newcommand{\inc}{\subset}

\newcommand{\Rt}{$\R$-tree}

\newcommand {\R} {{\mathbb {R}}}

\newcommand {\F} {{\mathbb {F}}}

\usepackage{srcltx}

\begin{document}

\title{Random subgroups, automorphisms, splittings       
}
\author{Vincent Guirardel, Gilbert Levitt}
\date{}

\maketitle

\begin{abstract}
 We show that, if $H$ is a random subgroup of a finitely generated free group $\F_k$, only inner automorphisms of $\F_k$ may leave $H$ invariant. A similar result holds for random subgroups of toral relatively hyperbolic groups, more generally of groups which are hyperbolic relative to slender subgroups. These results follow from non-existence of splittings over slender groups which are relative to a random group element. 
Random subgroups are defined using random walks or balls in a Cayley tree of $\F_k$.  
\end{abstract}

 \section{Introduction}

 When studying an automorphism of a group $G$, it is often useful to consider invariant subgroups. For instance, irreducible automorphisms of free groups are defined by considering invariant free factors   \cite{BH_tt}. 
  
 One may also fix a subgroup $H$ and consider   
the group $\Aut(G,H)$ of automorphisms of $G$ leaving $H$ invariant, or the group of automorphisms of $H$ extending to automorphisms of $G$. There may be many of those, for instance if $H$ is a free factor or a direct factor. The authors have proved that, conversely, if $H$ is a non-cyclic subgroup of a finitely generated free group $G$, and every automorphism of $H$ extends to $G$, then $H$ is  a free factor. 
 
 At the other extreme,     
Schupp proved in \cite{Schupp_inner} that any $H$ may be embedded into a group $G$ so that only inner automorphisms of $H$ extend. The proof uses small cancellation, so $G$ is defined in an explicit, but ad hoc, way. A starting point of the present paper is the idea that this non-extension phenomenon should be fairly common, and we express this by the following principle:
 
 {\it If $H$ is a very complicated subgroup of $G$, then very few automorphisms of $G$   leave $H$ invariant.}
 
But this principle is not  valid in full generality. For instance, if $G=\Z^n$, every subgroup 
is left invariant by many automorphisms.

  In any case, $\Aut(G,H)$ will always contain the group $\Inn_H(G)<\Inn(G)$ defined as the set of all conjugations by elements of $H$.
In this paper we prove:
 
 \begin{thm} [Theorem \ref{pasdaut}] \label{autrh}
Assume that $G$ is hyperbolic relative to a finite family $\calp$ of slender subgroups. 
If $H$ is a random subgroup of $G$,  
then  $\Aut(G,H)/\Inn_H(G)$ is finite.
\end{thm}

Recall that   $G$ is slender if $G$ and all its subgroups are finitely generated.    
To define  a random subgroup of $G$, we fix $p\ge1$ and we let $H$ be generated by $p$ independent random walks of length $n$ (see Definition \ref{rw} for details). The conclusion of the theorem then holds with probability going to 1 as $n\to\infty$. We rely on results of Maher-Sisto \cite{MaSi_random} about random walks, and our assumptions are the same as in their paper (Theorem \ref{autrh} would apply to subgroups generated by elements chosen randomly independently in balls, as in Theorem  \ref{autl} below, if the results of  \cite{MaSi_random} were   known to hold   
in that context).

We believe  
that $\Inn_H(G) $ is actually equal to $\Aut(G,H)$   when $G$ is torsion-free, but our methods do not allow us to prove it unless $G$ is a free group $\F_k$.

Let $X$ be a free basis of $\F_k$.
For the standard simple random walk on $\F_k$, associated to the uniform measure on $X^{\pm1}=X\cup X\m$,  
we show that generically  $\Inn_H(G) $ is precisely equal to $\Aut(G,H)$.
More precisely, we show that a random  subgroup $H$ is $\Aut$-malnormal in the following sense:

\begin{dfn}
A subgroup $H<G$ is \emph{$\Aut$-malnormal} if  any $\alpha\in\Aut(G)$ such that $\alpha(H)\cap H\neq \{1\}$ belongs to 
  $\Inn_H(G)$.  
\end{dfn}

 \begin{thm} [Theorem \ref{libre_RW}] \label{aut_rw_Fn}
Fix $k\ge2$ and $p\ge1$. 
Let $H=\grp{w_1,\dots,w_p}\subset \F_k$ be the subgroup generated by $p$ independent simple random walks $w_1,\dots, w_p$ of length $n$.

With probability going to $1$ exponentially fast as $n\to+\infty$, the subgroup 
$H$ is $\Aut$-malnormal.
In particular, $\Aut(G,H)=\Inn_H(G)$.
\end{thm}

There is a similar result if one  chooses the elements $w_1,\dots,w_p$ independently
randomly in the ball of radius $n$   (for the word metric associated to $X$). 

 \begin{thm} [Theorem \ref{libre_boule}] \label{autl}
Fix $k\ge2$ and $p\ge1$. With probability going to $1$ exponentially fast as $n\to+\infty$, the   subgroup   $H\inc \F_k$ generated by $p$ elements $w_i$ chosen randomly independently in the ball of radius $n$ is $\Aut$-malnormal (and therefore  $\Aut(G,H)=\Inn_H(G) $).
\end{thm}

 The proof  of these results uses Whitehead's peak reduction and equidistribution of subwords, so is specific to free groups.  
 
 The proof of Theorem \ref{autrh} uses the connection between automorphisms and 
 splittings (i.e.\   decompositions of $G$ as the fundamental group of a graph of groups). 
This is well-known in the context of (relatively) hyperbolic groups since Paulin's paper \cite{Pau_automorphismes} constructing an action of $G$ on an \Rt{}  for $G$   a hyperbolic group with $\Out(G)$ infinite (one then applies Rips's  theory of groups acting on \Rt s \cite{BF_stable} to get a splitting of $G$ over a virtually cyclic group).
 
 Another key idea of the present paper is a non-splitting principle. Recall that a splitting of $G$  is \emph{relative} to an element $h$ or a subgroup $H$ if $h$ (or $H$) is contained in  a conjugate of a vertex group (in other words, $h$ or $H$ fixes a point in the Bass-Serre tree).
 
Non-splitting principle: 
 {\it If $h$ is a very complicated element of a group $G$, it is universally hyperbolic: there is no splitting of $G$ relative to $h$
 (in other words, if $G$ acts on a tree with no global fixed point, then $h$ does not fix a point).
 }
 
Unfortunately this is false, even in free groups: given any $h\in \F_k$, there is an epimorphism $\F_k\onto \bbZ$ which   kills $h$, hence a splitting of $\F_k$ relative to $h$ (this splitting is over an infinitely generated group, but this may be remedied using standard approximation techniques).  By imposing conditions on edge groups,  however, one can get the following valid version of the non-splitting principle:

  \begin{thm} [Corollary \ref{premcor}]\label{pass}
 Let $G$   be a non-slender group which is  hyperbolic relative to a finite family of slender subgroups.
 \begin{itemize}
\item Let $w_n$ be given by a random walk on  $G$. 
With probability going to 1 as $n\to\infty$, there is no  splitting of $G$  over a slender subgroup  relative to $w_n$.

\item If $H$ is a random subgroup, then with probability going to 1 as $n\to\infty$ the group $H$ acts freely in every non-trivial $G$-tree with slender edge stabilizers (as in Theorem \ref{autrh}, $H$ is generated by $p$ independent random walks as in \cite{MaSi_random}).
 \end{itemize}
\end{thm}

JSJ decompositions of relatively hyperbolic groups are acylindrical (see   \cite{GL_JSJ}), and this allows us to apply the results of \cite{MaSi_random}. See  Corollary \ref{csa} for a similar result about torsion-free CSA groups, which also have acylindrical JSJ decompositions.

In the case of free groups, it was proved by Cashen-Manning \cite{CaMa_virtual} that $\F_k$ has no cyclic splitting relative to an element $g$ represented by a cyclically reduced word containing all reduced words of length 3 as subwords. The key technical result used to prove Theorem \ref{pass} is a generalization of this fact (see Theorem \ref{rehyp}). Our proof of Theorem \ref{rehyp} is self-contained and only uses basic Bass-Serre theory. It is inspired by ideas of Otal \cite{Otal_certaines}  and Cashen-Manning.

We also generalize Cashen-Manning's result in the following way.

\begin{thm}  [Theorem \ref{scindfn}]
 Let $k\ge2$ and $L\ge2$. Let $h\in\F_k$ be a cyclically reduced word containing all reduced words of length   $L$ as subwords. If $\F_k$ splits relative to $h$ over 
 a subgroup isomorphic to $\F_r$, then $r>(k-1)(L-2)$.
 \end{thm}
 
 The bound is sharp (see Proposition \ref{sharp}).

\section{Notations and conventions}
 
 We will always denote by $G$ a finitely generated group. We consider actions of $G$ on simplicial trees $T$ which are \emph{minimal} (there is no proper invariant subtree). We allow the \emph{trivial} action ($T$ is a point). We write $G_v$ for the stabilizer of a vertex $v$, and $G_e$ for the stabilizer of an edge $e$.
 
 We assume that $G$ acts \emph{without inversion} (if $g\in G$ leaves an edge invariant, it fixes its endpoints), and there is no \emph{redundant vertex} (if $v$ is a vertex of valence 2, there is $g\in G$ having $v$ as its unique fixed point).
 
We equip $T$ with the simplicial metric (every edge has length 1). A \emph{segment} $I$ is the geodesic joining two vertices. The \emph{translates} of $I$ are the segments $gI$, for $g\in G$.
 
 A \emph{splitting} of $G$ is an isomorphism of $G$ with the fundamental group of a graph of groups $\Gamma$, or equivalently an action on a tree $T$ (the Bass-Serre tree of the splitting). Splittings are always assumed to be non-trivial:  vertex groups are proper subgroups of $G$ (so $T$ is not a point).   
 A splitting $\Gamma$ is \emph{over} a subgroup $H$ if $H$ is an edge group. When $H$ is cyclic, $\Gamma$ is a \emph{cyclic splitting.}
 
  An element $g\in G$, or a subgroup $H\inc G$, is \emph{elliptic} in $T$ if it fixes a point in $T$. We then say that $T$ (or the corresponding splitting) is \emph{relative} to $g$ or $H$. If $g$ is not elliptic, it is \emph{hyperbolic} and has an \emph{axis}, a line on which it acts as a translation.
  
  A tree $\hat T$ is a \emph{refinement}
 of $T$ if one obtains $T$ from $\hat T$ by collapsing each edge belonging to some $G$-invariant set to a point.

  If $T,T'$ are two trees with an action of $G$, one says that $T$ is \emph{elliptic with respect to $T'$} if
every edge stabilizer of $T$ is elliptic  in $T'$. This implies (see Proposition 2.2 of \cite{GL_JSJ}) that $T$ has a refinement $\hat T$ which \emph{dominates} $T'$, in the sense that there exists a $G$-equivariant map from $\hat T$ to $T'$.

A group   $G$ is \emph{slender} if $G$ and all its subgroups are finitely generated.   Equivalently, whenever $G$ acts on a tree, there is a fixed point or an invariant line.
   A tree with an action of $G$ is \emph{slender} if its edge stabilizers are slender.

A subgroup $H\inc G$ is  \emph{almost malnormal} if  there exists $C$ such that $gHg\m\cap H$ has   cardinality at most $C$ for all $g\notin H$.

We denote by $\F_k$ the free group of rank $k$. Given a free basis $X$, a word $w=a_1\dots a_q$ with $a_i\in X^{\pm1}$ is (freely) \emph{reduced} if $a_{i+1}\ne a_i\m$ for $i=1,\dots,q-1$,  \emph{cyclically reduced} if in addition $a_{q}\ne a_1\m$. A \emph{subword} of $w$ is a word $a_i\dots a_j$ with $1\le i\le j\le q$. The \emph{$m$-prefix} of $w$ is the word $a_1\dots a_m$.

 If $w=a_1\dots a_q$ is reduced, its \emph{length} $ | w | $ is $q$.  In general, we identify a reduced word and the corresponding element of $\F_k$.

  Any finitely generated subgroup $H\inc \F_k$ has 
a   \emph{Stallings graph}  $\Theta$.  
 It has a base vertex $1$, its edges are oriented and labelled by elements of $X$.   The     elements of $H$  are precisely the words represented by immersed   paths  
with both endpoints $1$.  One may construct    $\Theta$ by letting $H$ act on the Cayley graph  $\Cay(\F_k,X)$, restricting to 
  the convex hull of the $H$-orbit of the base vertex,
and taking the quotient by the action of $H$.

\section{A non-splitting theorem}
 
 This section is devoted to the proof of Theorem \ref{rehyp}, which restricts the ways in which a group $G$ may split   relative to a complicated enough element $h$.  Theorem \ref{pass} will be  proved in Section \ref{nsr} by combining Theorem \ref{rehyp} with results by Maher-Sisto \cite{MaSi_random}.

 \begin{thm}  \label{rehyp}
 Let $S$ be a tree with an  action of $G$. Assume that $S$ is locally finite or  slender.  
  
 There exists a finite set $\cali$  of  segments $I_i\inc S$ of length at most 4 with the following property: 
if $h\in G$ is hyperbolic in $S$ and its axis contains a translate of each $I_i$, then $h$ remains hyperbolic in every non-trivial slender tree $T$
such that  $S$ is  elliptic with respect to $T$.
 \end{thm}
 
  \begin{rem} \label{petitstrucs}\ 
  \begin{itemize}
  \item Our implicit assumption that $S$ has no  {redundant vertex} is important  to bound the length of the $I_i$'s.
  \item
 If    $S$ has no vertex of valence 2, the $I_i$'s may be taken to be of length at most 3. Applying the theorem to the action of $\F_k$ on its Cayley tree yields Cashen-Manning's  theorem  \cite{CaMa_virtual}: $\F_k$ has no cyclic splitting relative to a cyclically reduced word $h$ containing all reduced words of length $\le3$ as subwords.
 \item
 The assumption that edge stabilizers of $T$ are slender may be weakened to saying that some edge stabilizer of $T$    
 is \emph{slender in $S$}:
it fixes a point  or leaves a line   invariant in $S$.
 \end{itemize}
 \end{rem}
 
 \begin{proof}  We may assume that $S$ is not a point or a line:
 the theorem is trivial if $S$ is a point, easy if $S$ is a line (in this case $S=T$).   
We may also assume that  there is  only   one orbit of edges in $T$.

 We start the proof by performing several constructions, starting with a tree $T$ 
as in the theorem.  
 The assumption that   
  $S$ is elliptic with respect to  $T$ implies  that
there exists a refinement $R$ of $S$ together with an equivariant map $f:R\ra T$ (see \cite[Proposition 2.2]{GL_JSJ} for instance).
We may assume that $f$ sends each vertex to a vertex, and each edge to  a point or an 
edge-path.

We fix an edge $e\subset T$, with midpoint $m$.
   We declare one component of $T\setminus\{m\}$ to be positive, the other negative.
We consider the set $M=f\m(m)$. It is $G_e$-invariant, contains no vertex, and $M/G_e$ is finite: if we subdivide $R$ so that the image of any edge is an edge or a point, the intersection of $M$ with a given $G$-orbit of edges consists of at most one $G_e$-orbit.

Let $\ell\inc R$ be any proper $G_e$-invariant subtree. There is one because 
  $G_e$  is slender, hence   fixes a point or leaves a line invariant, and $S$ is not a point or a line. For later use (in the proof of Theorem \ref{scindfn}), we do not assume yet   that $\ell$ is a point or a line.

We fix an integer $C$ such that $M$ is contained in the $C$-neighborhood $\ell_C$ of  $\ell$.
  Each component of $R\setminus\{\ell_C\}$ is mapped into a single component of $T\setminus\{m\}$, and we label it positive or negative accordingly.  
Any   ray $\rho\inc R$ having compact intersection with $\ell$ thus   inherits a sign (a ray is an isometric image  of $[0,+\infty)$).

\begin{lem} \label{2coul}
Let $w\in G$ be hyperbolic in $R$. Assume that   its axis  $A_w$ has compact intersection with $\ell$, and its ends have different signs. Then $w$ is hyperbolic in $T$.  
\end{lem}

\begin{proof}
Assume that $w$ fixes a vertex $x$ in $T$, say   in the positive component of $T\setminus\{m\}$. Let $y$ be any point of $A_w$,  
and $z=f(y)$. Replacing $w$ by $w\m$ if needed, we may assume that $w^ny$  goes to the  negative end of $A_w$ as $n\to+\infty$. Then $w^nz=f(w^ny)$ and $w^{n+1}z=w(w^nz)$ are in the negative component of $T\setminus\{m\}$ for $n$ large, and $wx=x$ with $x$ in the positive component  implies that $w$ fixes $m$.  We deduce that $w^nz$ is in the negative component for every $n$ (positive or negative), and both ends of the axis of $w$ in $R$ are negative, a contradiction.
\end{proof}

We say that a vertex $p\in\ell$ is a \emph{boundary vertex} if there is at least one edge incident on $p$ which is not contained in $\ell$
(if $\ell$ is a line, this means that $p$ is a vertex of $\ell$ having  valence $\ge3$ in $R$). 
Given a boundary vertex $p\in\ell$, we consider   rays $\rho$  with origin $p$ such that $\rho\cap\ell=\{p\}$. Such rays have a sign (positive or negative). 
We say that $p$ is \emph{positive} (resp.\  \emph{negative}) if  all rays  $\rho$ with origin $p$ such that  $\rho\cap\ell=\{p\}$ are positive (resp.\  negative),
and \emph{mixed} otherwise  (see Figure \ref{fig_couleurs}).

\begin{figure}[htbp]
  \centering
  \includegraphics{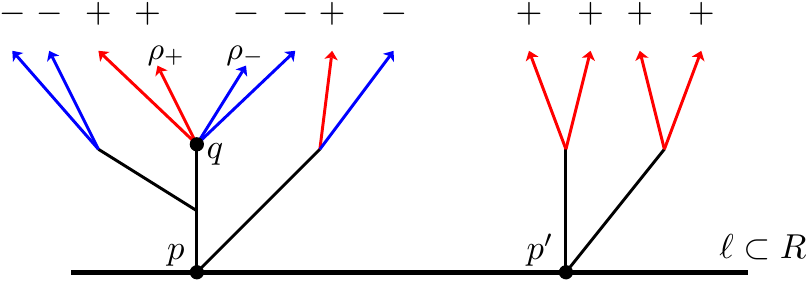}
  \caption{$p$ is a mixed vertex, $p'$ is positive.}
  \label{fig_couleurs}
\end{figure}

We shall distinguish two cases.

Case 1: $\ell$ contains a mixed vertex $p$.
In this case, we can find a vertex $q$ in $\ell_C$, with projection to $\ell$ equal to  $p$, such that 
 both a positive ray $\rho_+$ and a negative ray $\rho_-$ with origin $p$ pass through $q$, and the sign of a ray $\rho$ passing through $q$ only depends on the   edge   through which $\rho$ exits $q$
  (one can take for $q$ a point projecting to $p$, and furthest from $p$ with the property that there are rays of both signs
with origin $p$ passing through  $q$).
  
Case 2: $\ell$ has no mixed vertex (this cannot happen if $\ell$ is a point). 
In this case, each boundary vertex $p\in\ell$ inherits a sign, and both signs occur in $\ell$.

After these preliminary constructions, we recall that $S$ is locally finite or has slender edge stabilizers. 
We first suppose that it is locally finite. Vertex stabilizers of $S$ then contain an edge stabilizer with finite index, hence are elliptic in any $T$ as in the theorem. This implies that $S$ dominates $T$, so 
we may     take $R=S$ and view $\ell$ as a subtree of $S$. 

We   define a finite set $\cali$ by choosing a representative for each $G$-orbit of segments of length 4 in $S$. 
We consider  $T$ as in the theorem, $h$ whose axis in $S$ contains a translate of each segment of length 4, and we show that $h$  is hyperbolic in $T$.  We have distinguished two cases (depending on $T$ and $\ell$).

In case 1, some translate of the axis of $h$ passes through $q$ and contains the exit edges     of $\rho_+$ and $\rho_-$.   Lemma \ref{2coul} implies that some conjugate of $h$, hence also $h$ itself, is hyperbolic in $T$.

In case 2, 
 we recall that  we may take $\ell$ to be a line, so  $\ell $ contains   a positive $p_+$ and a negative $p_-$ which are  at distance 1 or   2   (at distance 1 if   all vertices of $\ell$  are boundary vertices); indeed, since $S$ has no redundant vertex and is not a line,
there are no adjacent vertices of valence 2 in $S$.
 The intersection of some translate of the axis of $h$ with $\ell$ is  precisely the segment $p_+p_-$, and hyperbolicity of $h$ follows from Lemma \ref{2coul}.

The argument when $S$ has slender edge stabilizers but is not locally finite is more complicated because there  may be infinitely many $G$-orbits of segments of length $\le4$. Also, we may have to take $R\neq S$ (and $R$ depends on $T$), but this issue  is   easily dealt with.

In order to  construct a suitable finite  family $\cali$ (independent of $T$) we use the case $k=3$ of the following lemma, whose proof we defer. 

\begin{lem} \label{seppair}
Let $k\ge1$. 
Let $H$ be a finitely generated group acting on an infinite set $\calx$ with finitely many orbits. 
Assume that    all point stabilizers $H_x$ are slender. The action of $H$ on $\calx$ extends to an action on a graph $\Delta$ with vertex set $\calx$ such that:
\begin{itemize}
\item there are finitely many $H$-orbits of edges in $\Delta$;
\item $\Delta$ is $k$-connected: it   cannot be disconnected by removing $k-1$ vertices \emph{(we use terminology from graph theory: 1-connected means connected, 3-connected means that there is no separating pair).}

\end{itemize}
\end{lem}

 Let $v$ be a vertex of $S$. 
We consider 
the action of its stabilizer $G_v$ (which is finitely generated because $G$ and edge stabilizers of $S$ are) on the link $L_v$ of $v$ in $S$ (the set of incident edges). Point stabilizers for this action are edge stabilizers of $S$, hence slender.     We apply the lemma with $k=3$. 
We get  a graph $\Delta_v$ with vertex set $L_v$ and no separating pair (if $L_v$ is finite, we let $\Delta_v$ be the complete graph with vertex set $L_v$). Since $G_v$ acts on $\Delta_v$, we may perform this construction $G$-equivariantly for all vertices $v$ of $S$.

Edges of $\Delta_v$ join  two elements of the link of $v$, we view them as segments of length 2 centered at $v$ in $S$.  Considering these segments for every $v$, we obtain  a family of segments  of length 2 consisting of finitely many 
$G$-orbits, and we include a representative of each orbit in $\cali$. 

We also consider representatives for $G$-orbits of edges of $  S$ bounded by two vertices   having valence at least 3, and for orbits of segments of length 2 whose midpoint has valence 2 (this is a finite set of orbits). For each such $\varepsilon$ we choose two extensions $\varepsilon_1\varepsilon\varepsilon_2$ and $\varepsilon'_1\varepsilon\varepsilon'_2$ of $\varepsilon$ to segments of length 3 or 4  respectively, with edges $\varepsilon_i\ne \varepsilon'_i$. We then add to $\cali$ the four segments  $\varepsilon_1\varepsilon\varepsilon_2$, $\varepsilon_1\varepsilon\varepsilon'_2$, $\varepsilon'_1\varepsilon\varepsilon_2$, $\varepsilon'_1\varepsilon\varepsilon'_2$.

Having constructed $\cali$, we now consider $T$ as in the theorem and $h$ whose axis in $S$ contains a translate of each $I_i$  in $\cali$, and we show that $h$ is hyperbolic in $T$.     
We first assume that $R=S$.  Since $G_e$ is slender
we may assume that $\ell$ is  a point or a line, and we consider the two cases introduced above.

In case 1,  we fix a mixed vertex $p\in\ell$. Recall that we have defined a vertex $q\in\ell_C$   projecting to $p$. Let $L_q$ be the link of $q$ in $S$ (the set of incident edges). We first define one or two special incident edges  at $q$. If $q\notin\ell$ (i.e.\ if $q\ne p$), the edge pointing towards $p$ is the only special edge. If $q\in\ell$ and $\ell$ is a line, both edges contained in $\ell$ are special. There is no special edge if $q\in\ell$ and $\ell$ is  a point. 

 Because of the way we defined $q$, non-special incident edges  $\zeta  $ at $q$  may be given a sign: they are   positive or negative, depending on whether rays with origin $p$ exiting $q$ through $\zeta$ are positive or negative,  and both signs occur. 
 
Using Lemma  \ref{seppair}, 
we have constructed a graph $\Delta_q$ with vertex set $L_q$ having no separating pair. 
This graph remains connected when we remove   the vertices corresponding to 
the (at most two) special edges. The remaining vertices correspond to incident edges  $\zeta$ at $q$ which are positive or negative.  Since   both   signs occur,   we may find a positive edge $\zeta_+$ and a negative edge $\zeta_-$ which are adjacent in $\Delta_q$. Because of the way we constructed $\cali$, some translate of the axis of $h$ in $S$ contains $\zeta_+\cup\zeta_-$ and Lemma \ref{2coul} implies that $h$ is hyperbolic in $T$, as required.

In case 2,  as in the locally finite case,  $\ell$ contains  a positive $p_+$ and a negative $p_-$ which are either adjacent or at distance   2 (separated by a vertex of valence 2). We included  four extensions of $\varepsilon=p_+ p_-$ in the $G$-orbit of  $\cali$, and one of them at least  intersects $\ell$ only along $p_+ p_-$. Some translate of the axis of $h$ contains this extension, and  $h$ is hyperbolic in $T$ by Lemma \ref{2coul}.

  To complete the proof of Theorem \ref{rehyp}, we need to consider the case when $R\ne S$. Let $\pi:R\to S$ be a collapse map.   Note that, if $\varepsilon$ is any open edge of $R$, both components of $R\setminus\{\varepsilon\}$ have unbounded image in $S$.
We define $\bar \ell=\pi(\ell)$, a point or a line,   and $\bar\ell_C$ its $C$-neighbourhood in $S$. 
The sign assignment of components of $R\setminus\{\ell_C\}$ induces one for components of $S\setminus\{\bar\ell_C\}$, with both signs appearing. Lemma \ref{2coul} applies in $S$  because if the axis of $w$ in $S$ has compact intersection with $\bar \ell$ and
its ends have two different signs, then the same holds for the axis of $w$ in $R$,
so the rest of the proof is the same as when $R=S$.
 \end{proof}
 
 \begin{proof}[Proof of Lemma \ref{seppair}]
The proof is by induction on $k$. If $k=1$, we just need $\Delta$ to be connected. This is easy to achieve, using finite generation of  $H$ and finiteness of $\calx/H$.

In the general case, we construct  $\Delta_1\inc \Delta_2\inc  \Delta_3=\Delta$
  by successively adding $H$-orbits of edges (each $\Delta_i$ is a graph with vertex set $\calx$ on which $H$ acts with finite quotient). At each step we specify a finite set of edges, and we obtain $\Delta_{i+1}$ from $\Delta_i$ by adding the $H$-orbits of these edges. 
  
As explained above, we may find a   connected graph $\Delta_1$. Given  an element   $x \in \calx$ (which we view as a vertex of $\Delta_1$),  
we view its link in $\Delta_1$ as the set of vertices adjacent to $x$. It  is itself a graph $L_x$ (possibly with no edge):  
there is an edge between $y$  and $y'$ in $L_x$ if and only if there is one in $\Delta_1$.
The stabilizer $H_x$ acts naturally on this graph $L_x$.

It is easy to check that $L_x/H_x$ is finite, so by induction we may add finitely many $H_x$-orbits of edges to $L_x$ in order   to make it  $(k-1)$-connected (if $L_x$ is finite, we make it a complete graph). We view these added edges as edges between elements of $\calx$, and since $\calx/H$ is finite we obtain a connected $\Delta_2$ with the property that all links of vertices are $(k-1)$-connected (or complete finite graphs).
 
 We now enlarge $\Delta_2$ in order to obtain $\Delta_3$ with the additional property that each edge is contained in a $(k-1)$-simplex (a complete subgraph with $k$ vertices). We claim that $\Delta=\Delta_3$ is then $k$-connected.
 
 Fix a subset $\calx_0$ of cardinality $k-1$ in $\calx$. 
  We must be able  to join any two  vertices $x,y$ in $\calx\setminus \calx_0$  by a path in $\Delta$ avoiding $\calx_0$. Since $\Delta$ is connected, we may find a path from $x$ to $y$. It suffices to consider the case when this path is of the form $xz_1\dots z_py$ with the $z_i$'s distinct elements of $\calx_0$.

  First suppose $p=1$, so that $x$ and $y$ belong to the link of $z_1$, which is $(k-1)$-connected. The intersection of this link with $\calx_0$ has cardinality at most $k-2$, so we may join $x$ to $y$ in the complement of  $\calx_0$. If $p\ge2$, we consider the edge $z_1z_2$. It is contained in a $(k-1)$-simplex, which has $k$ vertices so contains a vertex $z\notin \calx_0$. We then replace the path $xz_1\dots z_py$ by the concatenation of $xz_1z$ and $zz_2\dots z_py$ and use induction on $p$.
  \end{proof}

\section{Splittings of free groups}

 We view $\F_k$ as the set of reduced words on a set $X$ of cardinality $k$. 
 \begin{thm}  \label{scindfn}
 Let $k\ge2$ and $L\ge2$. Let $h\in\F_k$ 
 be a cyclically reduced word containing all reduced words of length   $L$ as subwords. If $h$ is elliptic in a splitting of   $\F_k$ over a subgroup isomorphic to $\F_r$, then $r>(k-1)(L-2)$.
 \end{thm}
 
 In other words: if $h$ is complicated, all splittings relative to $h$ are over groups of large rank.
 
When $L=2$, the theorem says that $h$ is not contained in a proper free factor, a result   due to Whitehead. When $L=3$ there is no splitting of $\F_k$ relative to $h$ over $\F_r$ if $r\le k-1$  (the case $r=1$ is due to Cashen-Manning \cite{CaMa_virtual}).

 \begin{proof}
 We argue as in the proof of Theorem \ref{rehyp}, with $S$ the Cayley graph  Cay$(\F_k,X)$
 (a locally finite tree),  
 $T$  the Bass-Serre tree of a splitting over $\F_r$, 
and  $e$ an edge of $T$ (note that 
 $G_e\simeq \F_r$ is not slender if $r\ge2$). We let $\ell\inc S$ be   any point in $S$ if $G_e$ is trivial, the minimal $G_e$-invariant subtree otherwise (it is a proper subtree because $G_e$ has infinite index: otherwise $G$ would fix a point in $T$ and the splitting would be trivial).   
 
 As in the proof of Theorem \ref{rehyp}, we distinguish case 1 and case 2. In case 1 (there exists a mixed vertex in $\ell$), no new  argument is needed   
since we assume $L\ge 2$. In case 2,  
all boundary vertices of $\ell$ are positive or negative,
but we can no longer find boundary vertices $p_+$, $p_-$ with distance at most 2 (this required $\ell$ to be a line). In fact, if $h$ as in the theorem is elliptic in $T$, any boundary vertices $p_+$, $p_-$ of opposite signs must be at least $(L-1)$-apart: otherwise the axis of a conjugate of $h$ intersects $\ell$ precisely in the segment $p_+p_-$, so $h$ is hyperbolic in $T$ by Lemma \ref{2coul}.

Choose a pair of boundary vertices of opposite signs $p_+, p_-\in\ell$ whose distance $D$ is minimal. We have seen $D\ge L-1$.
Every vertex of $\ell$ which is not a boundary vertex has valence $2k$ in $\ell$, so
all vertices between $p_+$ and $p_-$ have valence $2k$ in $\ell$. 
The quotient map from $\ell$ to $\ell/G_e$ (a regular covering with group $G_e$) is injective on the segment $p_+p_-$ because the sign assignment is $G_e$-invariant. The quotient graph $\ell/G_e$ therefore has at least $D-1$ vertices with valence $2k$. Since it has no vertex of valence 1 and its fundamental group has rank $r$, we get $r> (k-1)(D-1)\ge (k-1)(L-2)$. 
 \end{proof}
 
The theorem may be generalized, for instance to the following statement.  

\begin{thm} \label{virtlib}
Let $G$ be a finitely generated group acting on a tree $S$ with finite stabilizers (so $G$ is virtually free). Let $L$ be an integer, and $\calh$   a family of elements of $G$ such that every segment of length $\le L$ in $S$ is contained in a translate of the axis of  an element of $\calh$. 

If $G$ splits relative to $\calh$ over a subgroup which is virtually $\F_r$, then $r\ge L/4$ (and $r\ge L/2$ if $S$ has no vertices of valence 2). \qed
\end{thm}

We leave details to the reader.

We now show that the bound in Theorem 
\ref{scindfn} is optimal (at least for $L$ even).

\begin{prop} \label{sharp}
{For each $k\ge2$ and each even $L=2i\ge2$, there is  a splitting of $\F_k$ over a group of rank $r=(k-1)(L-2)+1$ relative to a cyclically reduced $h$ containing all reduced words of length $L$. }
 \end{prop}
 
  \begin{rem}
 We are not sure of optimality for $L$ odd. For instance,   there seems to be no   splitting of $\F_2$ over $\F_2$ relative to an $h$ containing all reduced words of length 3.
 \end{rem}

 \begin{figure}[ht!]
  \centering
  \includegraphics{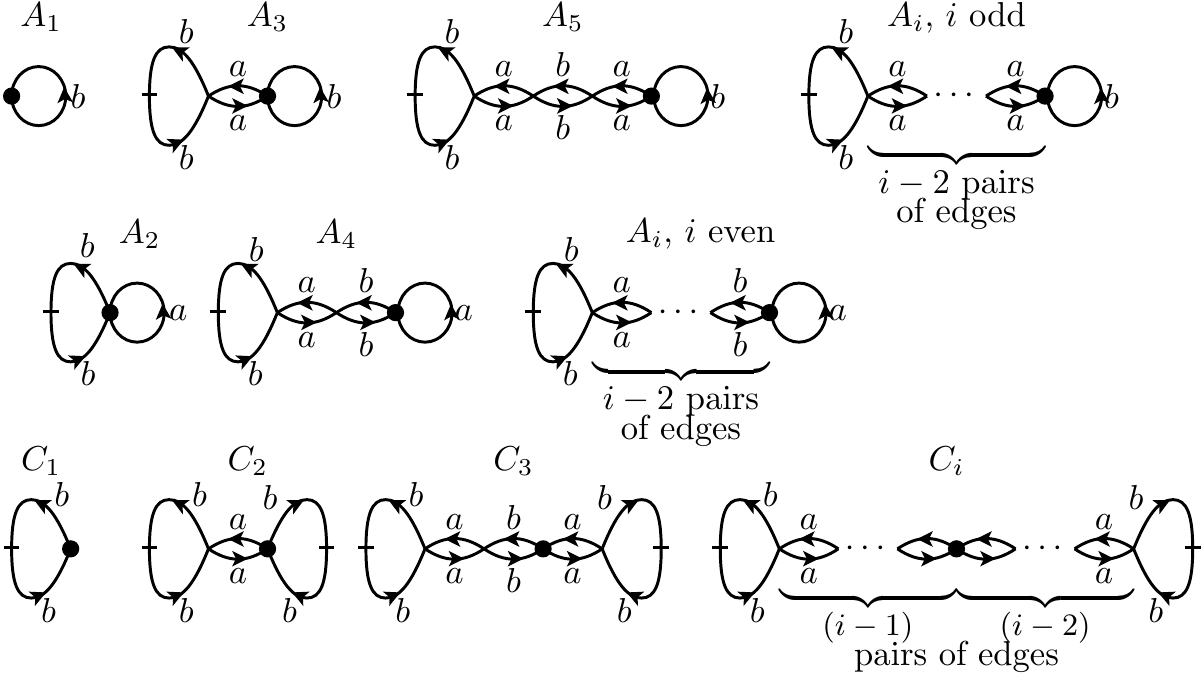} 
  \caption{The Stallings graphs of the subgroups $A_i,C_i$ of $\F_2$ (the bullet represents the base vertex).}
  \label{fig_stallings}
\end{figure}

 \begin{lem}[see Figure \ref{fig_stallings}]
 Fix  $k\ge 2$. For $i=1,2,\dots$,  there are subgroups $A_i,C_i$,
 and  splittings of $\F_k$ as $A_i*_{C_i}A_{i+1}$,
 such that:
 
 \begin{itemize}
 \item   
$A_i$ has rank $i(k-1)$;
 
 \item  $C_i$   
  has index 2 in $A_i$, hence  has rank $2i(k-1)-1$;

 \item all reduced words of length $i$ may be read as labels of paths in the Stallings graph of $A_{i+1}$   starting at the base  vertex.
 \end{itemize}
 \end{lem}
 
 The last item ensures that   
 $A_{i+1}$  
 contains   cyclically reduced elements containing all reduced words of length $L=2i$, so the proposition follows from the lemma.
 
 The Stallings graphs of the groups $A_i$ and $C_i$ are pictured on Figure \ref{fig_stallings} in the case of $\F_2=\grp{a,b}$. 
  One easily checks that $C_i\subset C_{i+1}$, and that $A_{i+1}=\grp{A_{i-1},C_i}$.
The initial splitting is $A_1*_{C_1} A_2=\grp{b}*_{\grp{b^2}} \grp{a,b^2}$.
  For all $i$
one obtains the splitting $A_{i}*_{C_{i}} A_{i+1}$ from
$A_{i}*_{C_{i-1}} A_{i-1}$ by folding $C_i<A_i$ along the edge, thus replacing $C_{i-1}$ by $C_i$
and $A_{i-1}$ by $\grp{A_{i-1},C_i}=A_{i+1}$.
 The reader may check that these splittings have the required properties.
For $k>2$, one adds $k-2$ loops labelled by the extra generators at each vertex of each Stallings graph.

\section{Random walks  (after Maher-Sisto \cite{MaSi_random})}

\begin{dfn} [Random subgroup, random element]\label{rw}
Let $G$ be a finitely generated group. Let $\mu$ be a probability measure on $G$ whose support is finite and generates $G$ as a semigroup. We fix $p\ge1$, and we consider a subgroup $H\inc G$ generated by $p$ elements $w_{1,n},\dots,w_{p,n}$ arising from   independent random walks of length $n$ generated by $\mu$. We call $H$   a \emph{random  subgroup} of $G$. When $p=1$, we call $w_n=w_{1,n}$   a \emph{random element}.
\end{dfn}

\begin{rem} The assumptions on $\mu$ and $H$ may be weakened to those of  \cite{MaSi_random}. 
\end{rem}

 Recall that $G$ acts   \emph{acylindrically} on a tree $S$ if there exist numbers $K$ and $C$ such that stabilizers of segments of length $K$ have cardinality at most $C$ (this is sometimes called almost acylindrical, and agrees with the general definition of acylindricity given in \cite{MaSi_random}).

The following theorem will ensure that Theorem \ref{rehyp} applies to non-trivial elements of random subgroups if $S$ is   
acylindrical.

\begin{thm}\label{MaSi}
Assume that $G$ is not virtually cyclic and acts
  acylindrically on a non-trivial tree $S$. Let $\cali$ be a finite family of segments $I_i\inc S$. Let $H=\grp{w_{1,n},\dots,w_{p,n}}$ be a random subgroup as in Definition \ref{rw}.

With probability going to 1 as $n\to\infty$, the group  $H$ is freely generated by $w_{1,n},\dots,w_{p,n}$, the action of $H$ on $S$ is free, and the axis of any non-trivial $h\in H$ contains a translate of each $I_i$.

\end{thm}

 We explain how to derive this theorem from \cite{MaSi_random}. Since the action on $S$ is acylindrical and $G$ is not virtually cyclic, $S$ is irreducible (there is no fixed point, no fixed end, no invariant line), so the action is non-elementary in the sense of \cite{MaSi_random}. By the main theorem of \cite{MaSi_random}, the $w_{i,n}$'s freely generate $H$ with probability going to 1, and $HE(G)$ is hyperbolically embedded in $G$ (with $E(G)$ the maximal finite normal subgroup of $G$).

Choose a basepoint $x_0\in S$, and fix a hyperbolic element $g\in G$ such that some fundamental domain for the action of $g$ on its axis  contains a translate of each $I_i$ 
(one finds such a $g$ by applying    Lemma 4.3 of \cite{Pau_Gromov} inductively). 

We first consider the case $p=1$ and we let $\gamma_n$ be the segment between $x_0$ and $w_nx_0$.
Applying Proposition  10 (4) of \cite{MaSi_random} with $\varepsilon=1/4$ and $L$ large with respect to the constant $K_0$ and the translation length of $g$, we deduce that the middle half of $\gamma_n$ contains a translate of each $I_i$ with probability going to 1. By Proposition 10 (5) of \cite{MaSi_random}, this also holds for the axis of $w_n$.

For $p>1$, we consider the smallest $H$-invariant subtree $S_H\subset S$ containing $x_0$. It follows from Propositions 30 and 32 of \cite{MaSi_random} that, with probability going to 1, the action of  $H$  on $S_H$ is free, and the quotient looks like a rose:   it is the union  of a central tree $C$ with diameter $<\varepsilon n$ (for some arbitrarily  small $\varepsilon>0$) and $p$ arcs $\theta_1,\dots,\theta_p$ of length   $>(L-\eps) n$  attached to $C$ (with $L>0$ the drift of the random walk), and  moreover   the image of the axis of $w_{i,n}$ in $S_H/H$ is the union of $\theta_i$ with an arc contained in $C$ (compare  the  {central tree property}, see e.g.\  \cite[section 3.1]{BMNVW}  and Subsection \ref{sctp}). The   image of the axis of any non-trivial $h\in H$  in $S_H/H$  contains one of the $\theta_i$'s, and the result follows.

\section
{Non-splitting relative to random elements}
\label{nsr}

One basic theme of this paper is that a group  has no non-trivial splitting relative to a random element (or a random subgroup). 
As explained in the introduction, 
 one must   impose restrictions on the edge groups of the splitting. In the case of $\F_k$, combining Theorems \ref {scindfn} and \ref{MaSi} yields:
 
 \begin{thm}
Fix $r\ge1$.   Let $w_n\in\F_k$ be   a random element     as in Definition \ref{rw}.   With probability going to 1 as $n\to\infty$, there is no splitting of $\F_k$ relative to $w_n$ over a group of rank at most $r$.
\end{thm}

 \begin{proof} Choose $L$ such that $r\le(k-1)(L-2)$. Apply Theorem \ref{MaSi} to the action of $\F_k$ on its Cayley tree. With probability going to 1, the axis of $w_n$ contains (translates of) all segments of length $  L$, so $w_n$ is hyperbolic in every splitting over a group of rank $\le r$ by Theorem \ref {scindfn}. 
 \end{proof}

 In general, we get:

\begin{thm} \label{main}
Assume that $G$ is not virtually cyclic and acts
  acylindrically on a non-trivial tree $S$ which is locally finite or  slender.    Let $H$ be a random subgroup as in Definition \ref{rw}.
  With probability going to 1 as $n\to\infty$, the group $H$ acts freely in every non-trivial slender tree  $T$  
 such that  $S$ is elliptic with respect to $T$.
\end{thm}

 \begin{proof}
 Use Theorem \ref{MaSi} and apply Theorem  \ref{rehyp} to all non-trivial elements of $H$.
  \end{proof}
  
  We refer the reader to \cite{GL_JSJ}   
  for details about the JSJ decompositions used in the next   
  results. 
  
   \begin{cor}\label{cor_fp}
Assume that $G$ is finitely presented,  not  
 virtually $\bbZ^2$,  
 and the JSJ decomposition of $G$ over virtually cyclic  subgroups is acylindrical. Let $w_n$ be a random element of $G$. 
With probability going to 1 as $n\to\infty$, there is no  splitting of $G$  over a virtually cyclic subgroup  relative to $w_n$.
\end{cor}

 \begin{proof}
Let $S$ be an acylindrical JSJ tree over virtually cyclic subgroups.   By definition of the JSJ decomposition, 
$S$ is elliptic in every tree $T$ with virtually cyclic   edge stabilizers,  so the result follows by    applying Theorem \ref{main} to $S$, provided that $S$ is not a trivial tree (a point).

 If $S$ is a point, there are two cases: rigid or flexible (see  \cite{GL_JSJ}, Definition 2.14). In the rigid case, $G$ does not split over a virtually cyclic group so the theorem is empty. In the flexible case, it follows from Theorem 6.2 and Proposition 6.38 of \cite{GL_JSJ} that  either $G$ is  
 virtually $\bbZ^2$,  contrary to our hypothesis,
or
   $G$ maps onto the fundamental group of a closed hyperbolic 2-orbifold $\Sigma$ with finite kernel (it is QH with finite fiber, see \cite[Theorem 6.2]{GL_JSJ}).

For simplicity we assume that $\Sigma$ is a surface rather than an orbifold. We apply \cite{MaSi_random} to the action of $G$ on the hyperbolic plane $H^2$ 
(viewed as the universal cover of $\Sigma$). 
We fix a closed geodesic $\gamma$ which fills $\Sigma$. 
By \cite{MaSi_random}, there exists a constant $K$ such that, for any compact segment $A\subset\gamma$, 
 the axis of $w_n$ is $K$-close to some translate of $A$ with high probability. 
If $A$ is chosen long enough, this implies that the closed geodesic representing (the conjugacy 
class of the image of) $w_n$ meets every simple closed geodesic $\delta$, so $w_n$ is hyperbolic in the splitting of $G$ dual to $\delta$. 
The result follows since every splitting of $G$ over a slender group is dual to a  simple closed   geodesic of $\Sigma$ (see e.g.\ Sections 5.1.2 and 5.2   
of \cite{GL_JSJ}).
  \end{proof}

  \begin{rem} 
    The same argument shows that, if $G$ is finitely presented, not slender, and its slender JSJ decomposition  is acylindrical,
then, with probability going to 1 as $n\to\infty$, there is no  splitting of $G$  over a slender 
subgroup  relative to $w_n$.
  \end{rem}

 \begin{cor} \label{premcor}
 Let $G$ be   a non-slender group which is hyperbolic relative to a finite family of slender subgroups $P_i$.
 \begin{itemize}
\item Let $w_n$ be a random element of $G$. 
With probability going to 1 as $n\to\infty$, there is no  splitting of $G$  over a slender subgroup  relative to $w_n$.
\item 
If $H$ is a random subgroup, then with probability going to 1 as $n\to\infty$ the group $H$ acts freely in every non-trivial tree with slender edge stabilizers.
\end{itemize}
\end{cor}

\begin{proof}
 We let $S$ be a JSJ tree over slender groups relative to the parabolic subgroups $P_i$ (which we may assume not to be virtually cyclic), see \cite{GL_JSJ}, in particular Theorem 9.18 and Corollary 4.16. It is 2-acylindrical (stabilizers of segments of length 3 are finite with bounded cardinality), and its edge stabilizers are elliptic in every slender tree $T$.   As in the previous proof, 
 we   apply Theorem \ref{main} to $S$.   In the flexible case,  $G$ is QH with finite fiber by Theorem 9.18 of \cite{GL_JSJ}.
\end{proof}

 \begin{cor}\label{csa}
 Let $G$ be a torsion-free CSA group.  
Let $w_n$ be a random element of $G$. 
With probability going to 1 as $n\to\infty$, there is no  splitting of $G$  over a finitely generated abelian subgroup   relative to $w_n$.

\end{cor}

 Recall that a group is CSA if its maximal abelian subgroups are malnormal.

\begin{proof}
If all abelian subgroups of $G$ are finitely generated (hence slender), the proof is the
same as  that of  the previous corollary, using a JSJ decomposition over abelian groups relative to all non-cyclic abelian subgroups (see Theorem 9.5 of \cite{GL_JSJ}).

In general, we apply Corollary 9.1 of \cite{GL_JSJ} with $\cala$ the family of all finitely generated abelian subgroups,   $\cals$ the family of all abelian subgroups (note that conditions (4b) and (4c) of the corollary are satisfied), and $\calh=\es$. We obtain a tree $S=(T_a)^*_c$ which is a JSJ tree over $\cala$ (hence has finitely generated edge stabilizers) relative to all non-cyclic abelian subgroups. It is compatible with every tree $T$ with edge stabilizers in $\cala$, in particular it is elliptic with respect to $T$, and we can argue as before.
\end{proof}

\begin{rem}
 If $G$ is finitely presented, there is no splitting $T$ of $G$ relative to $w_n$ over \emph{any} abelian subgroup. To see this, we apply  Theorem  6.36 of \cite{GL_JSJ},  with  $\cala$ the family of \emph{all} abelian subgroups and $\calh=\es$. By Theorem 2.20 of \cite{GL_JSJ}, there is a JSJ tree $S$ over $\cala$ with finitely generated (hence slender) edge stabilizers. The edge stabilizers of $T$ do not have to be slender, but they are slender in $S$ and we use Remark \ref{petitstrucs}.
\end{rem}

\begin{cor}
Assume that  $G$ has infinitely many ends,
and let $w_n$ be a random element of $G$. 
With probability going to 1 as $n\to\infty$, there is no  splitting of $G$  over a slender subgroup   relative to $w_n$.
\end{cor}

\begin{proof}
 Apply Theorem \ref{main} with $S$ a tree with finite edge stabilizers  (such an $S$ is elliptic with respect to any $T$).  
\end{proof}

\begin{rem}
 The result  remains true if the edge group of the splitting is only assumed not to split over a finite group.
\end{rem}

\begin{cor}
Assume that  $G$ splits over a slender almost malnormal subgroup $H$, and let $w_n$ be a random element of $G$. 
With probability going to 1 as $n\to\infty$, there is no  splitting of $G$  over a slender subgroup   relative to $w_n$ and $H$.
\end{cor}

Recall that $H$ is almost malnormal if  there exists $C$ such that $gHg\m\cap H$ has   cardinality at most $C$ for all $g\notin H$.

\begin{proof}
 Apply Theorem \ref{main} with $S$  the  given splitting of $G$ over $H$. Almost malnormality of $H$ implies that it is acylindrical. 
 Edge stabilizers of $S$ are conjugate to $H$, hence elliptic in  any tree $T$ relative to $H$. 
\end{proof}

\section{Automorphisms}

We now prove several   results saying that    
  few automorphisms of a given group $G$ leave a random subgroup invariant.   We  shall consider relatively hyperbolic groups, before focusing on the specific case of free groups.

Before doing that, we   note the following consequence of \cite{GL6} (Theorem 7.6 or 7.14).

\begin{thm} Let $G$ be a  
hyperbolic group, and $g\in G$ an element of infinite order. If $G$  does not split 
relative to $g$ over a virtually cyclic group with infinite center,
the stabilizer of $g$ in $\Aut(G)$ 
is virtually cyclic (it is virtually generated by 
 the conjugation by $g$).  \qed
 \end{thm}

In particular, using \cite{CaMa_virtual}, we see that, if $w\in\F_k$ is represented by a cyclically reduced word containing every reduced word of length 3, then the stabilizer of  $w$ in $\Aut(\F_k)$ is virtually cyclic.

\subsection{Relatively hyperbolic groups}

\begin{thm} \label{pasdaut}
Assume that $G$ is hyperbolic relative to a finite family $\calp$ of slender subgroups. 
Let $H=\grp{w_{1,n},\dots,w_{p,n}}$ be a random subgroup as in Definition \ref{rw}.
With probability going to 1 as $n\to \infty$, the subgroup $\Inn_H(G)\in\Aut(G)$ generated by conjugations by elements of $H$ has finite index in the group $\Aut(G,H)$
  of automorphisms of $G$ leaving $H$ invariant. 

\end{thm}

The proof requires a lemma.  
 A group $P$ is \emph{small} if it does not contain $\F_2$.

\begin{lem}\label{scind}
Let  $\calp_0$ be a finite family of small finitely generated subgroups $P_i$ which are not virtually cyclic.
Assume that $G$ is hyperbolic relative to $\calp_0$,  and also relative to $ \calp_0\cup\{ H\}$ with $H$ infinite and finitely generated. Also assume that non-small subgroups of $H$ have finite centralizer (in $H$ hence also in $G$).

If $\Inn_H(G)$
 has infinite index in 
$ \Aut(G,H)$,
then $G$ has a splitting over a small group,  and this splitting is relative to the $P_i$'s and to some $H_0\inc H$ which is equal to $H$ or contains $\F_2$.

\end{lem}

If  
  $\calp_0=\es$ and
 $H=\{1\}$, the lemma reduces to the standard statement  that a hyperbolic group $G$ with $\Out(G)$ infinite splits over a small group.

\begin{proof}
Our assumptions imply that the group $\Out(G,\calp_0\cup\{ H\})\inc\Out(G)$ of 
  outer automorphisms sending     the $P_i$'s and $H$ to conjugates is infinite. This is because  $H$ is infinite and almost malnormal  (so   inner automorphisms of $G$ leaving $H$ invariant  are conjugations by elements of  $ H$), and every automorphism of $G$ maps $P_i$ to a conjugate of some $P_j$ (so $\Out(G,\calp_0)$ has finite index in $\Out(G)$).

We view  $G$ as hyperbolic relative to   $ \calp_0\cup\{ H\}$, and we apply Corollary 
 7.13 of \cite{GL6} with $\calp= \calp_0\cup\{ H\}$ and $\calh$ empty. We get a graph of groups decomposition  $\Gamma$  of $G$ relative to the $P_i$'s and $H$, with edge groups  small or contained in $H$ (up to conjugacy).  The lemma is proved if some edge group of $\Gamma$ is small, so we assume that all edge groups are conjugate to subgroups of $H$.

The assumption about centralizers implies that the group of twists of $\Gamma$ is finite  (see \cite{GL6} for definitions not given here).
By  \cite[Corollary 7.13]{GL6}, infiniteness of $\Out(G,\calp_0\cup\{ H\})$ implies that $\Gamma$ has a vertex group  $G_v$   with $\Out(G_v;\Inc_v^{(t)})$ infinite: $G_v$ has infinitely many outer automorphisms acting on  incident edge groups as conjugations by elements of $G_v$. 

Corollary 7.13 of \cite{GL6} also implies that 
 $G_v$ is a maximal parabolic subgroup. It
  has to be conjugate to $H$: otherwise  it  would be conjugate to some $P_i$, and incident edge groups would be small.
  
 The automorphisms in $\Out(G_v;\Inc_v^{(t)})$ extend to $G$ and we get that   
 $\Out(G,\calp_0 )\cap\Out(G,\{H_0\}^{(t)})$ is infinite for some non-small $H_0\inc H$ (an incident edge group at $v$).    We now view $G$  as hyperbolic relative to $\calp_0$ only, and we apply Corollary 7.13 of \cite{GL6} with $\calp=\calp_0$ and $\calh=\{H_0\}$. We get a splitting of $G$ which is relative to the $P_i$'s and $H_0$, over a group which is virtually cyclic or contained in some $P_i$ (up to conjugacy), hence small.
\end{proof}

\begin{proof}[Proof of Theorem \ref{pasdaut}]
First assume that $G$ is torsion-free. 
By \cite{MaSi_random}, with probability going to 1, the group $H$ is free and malnormal.  In order to apply Lemma \ref{scind}, we just need to check that 
  $G$ is hyperbolic relative to $\calp\cup\{H\}$ (we may assume with no loss of generality that no $P_i\in\calp $ is virtually cyclic).
  
Recall \cite[Proposition 4.28]{DGO_HE} that $G$ being hyperbolic relative to $\calp$ is equivalent to $\calp$ being hyperbolically embedded in $(G,X)$ (with $X$ a finite generating set of $G$ relative to $\calp$).  
  By Theorem 5 of \cite{MaSi_random}, the group $H$ is quasi-isometrically embedded and geometrically separated in Cay$(G,X\cup \calp)$ with probability going to 1, hence (\cite[Theorem 3.9]{AMS_commensurating}) $\calp\cup\{H\}$ is hyperbolically embedded in $(G,X)$,
i.e.\ $G$ is  indeed hyperbolic relative to $\calp\cup\{H\}$ (since $X$ is finite).

If Theorem \ref{pasdaut} is false, Lemma \ref{scind} provides  a splitting of $G$   over a slender subgroup which contradicts Corollary \ref{premcor}.

We now allow torsion. 
Let  $E(G)$ be  the maximal finite normal subgroup of $G$,  and  $\bar H=HE(G)$. With probability going to 1 it is virtually free (hence satisfies the condition on centralizers in Lemma \ref{scind}) and  $G$ is hyperbolic relative to $\calp\cup\{\bar H\}$ as above. 
Since $\Aut(G,H)\inc \Aut(G,\bar H)$ (because $E(G)$ is characteristic) and $H$ has finite index in $\bar H$, it suffices to prove
that $\Inn_{\ol H}$ has finite index in $\Aut(G,\bar H)$.
If this does not hold, Lemma \ref{scind} provides a splitting relative to some   infinite subgroup $\bar H_0\inc\bar H$.
 Since $\bar H_0\cap H$ is non-trivial and fixes a point in this splitting, this contradicts Corollary \ref{premcor}. 
\end{proof}

\subsection{Free groups}

Let $\F_k$ be a free group of rank $k\geq 2$.
 We denote by $\ad_h$ the inner automorphism $g\mapsto hgh\m$.

\begin{dfn}
  A subgroup $H<\F_k$ is \emph{$\Aut$-malnormal} if, for any
  $\alpha\in \Aut(F_k)$ such that $\alpha(H)\cap H\neq \{1\}$, there
  exists $h\in H$ such that $\alpha=\ad_h$.
\end{dfn}

Clearly, if $H$ is $\Aut$-malnormal and non-trivial, then the only automorphisms of $G$ preserving $H$ are   conjugations by elements of $H$.
With the notation of Theorem \ref{pasdaut}, this says that $\Aut(G,H)=\Inn_H(G)$ (exactly, not up to finite index).

In this section, we prove two results saying that random subgroups of the free group are $\Aut$-malnormal,
one for groups generated by elements chosen   randomly independently in a ball of large radius,
and one for groups generated by elements coming from independent simple random walks.

We fix a free basis $X$ of $\F_k$.
We view elements   $g\in\F_k$ as reduced words in $X^{\pm1}$, and we write $ | g | $ for the length of $g$.
  Balls are defined using the  generating set $X^{\pm1}$, and we consider the simple random walk  
where $w_{n}=s_1\cdots s_n$ with $s_1,\dots,s_n$ chosen randomly and independently  in $X^{\pm1}$  (equipped with the uniform measure). 

 We say that an event occurs \emph {with probability going to $1$ exponentially fast as $n\to+\infty$} if the probability that it does not occur is bounded by $C\kappa^{-n}$ for some constants $C,\kappa>0$.

 \begin{thm} \label{libre_boule}
Fix $k\ge2$  and $p\ge1$. With probability going to $1$ exponentially fast as $n\to+\infty$, the   subgroup   $H\inc \F_k$ generated by $p$ elements $w_i$ chosen randomly independently in the ball of radius $n$ is $\Aut$-malnormal.
\end{thm}

See Theorem 8.5 and Proposition 8.7 of \cite{KSS_generic} for the case $p=1$. 

 \begin{thm} \label{libre_RW}
Fix $k\ge2$ and $p\ge1$. 
Let $H=\grp{w_{1,n},\dots,w_{p,n}}$ be the subgroup generated by $p$ elements $w_{1,n},\dots,w_{p,n}$ arising from independent simple random walks of length $n$ in $\F_k$.

With probability going to $1$ exponentially fast 
as $n\to+\infty$, the subgroup $H$ is $\Aut$-malnormal.
\end{thm}

\begin{cor}
  In the setting of Theorems \ref{libre_boule} and \ref{libre_RW}, 
the only automorphisms of $G$ preserving $H$ are   conjugations by elements of $H$.\qed
\end{cor}

Both theorems are special cases of the following general statement. 

\begin{prop}\label{simple}
 Fix $k\ge2$ and $p\ge1$. 
Let $ {w_{1,n},\dots,w_{p,n}}$ be independent random variables
 in $\F_k$ satisfying the following conditions:
\begin{itemize}
 \item (Radial symmetry) Given $n$ and $i$, the probability that $w_{i,n}=g$ only depends on the length of the element $g\in\F_k$.

 \item (Positive drift) There exists $L>0$ such that, for each $i$, the probability that $ | w_{i,n} |> Ln$ goes to 1 exponentially fast as $n\to+\infty$.

 \item (Subexponential growth) 
 For any $\theta>0$,   
the probability that   
$| w_{i,n} | \leq e^{\theta n}$ goes to 1 exponentially fast as $n\to+\infty$. 
 \end{itemize}

Then, with probability going to $1$   exponentially fast 
as $n\to+\infty$, the subgroup $H=\grp{w_{1,n},\dots,w_{p,n}}$ is $\Aut$-malnormal.

\end{prop}

This proposition clearly implies Theorems \ref{libre_boule} and \ref{libre_RW} (it is well-known that the simple random walk on $\F_k$ has positive drift $1-\frac1k$). 

We  shall now prove the proposition.  For simplicity, we sometimes write \emph{generically} to mean with probability going to 1 exponentially fast as $n\to+\infty$.

Many arguments already appear in \cite{KSS_generic} or  \cite{BMNVW}, but 
  we do not have information about the distribution of the lengths $ | w_{i,n} | $, so we will have to  use a Fubini-type argument, working with spheres rather than balls (this is made possible by radial symmetry).  
  
  More precisely, let $C,\theta,\kappa$ be positive numbers. 
\emph{Suppose that, given $n$ and numbers $A_1,\dots,A_p$ with   $Ln\le A_i\le C e^{\theta n}$, 
the probability that words $w_1,\dots,w_p$ with $ | w_i | =A_i$ 
(chosen uniformly independently on spheres of radius $A_i$) satisfy a given property 
is at least $1-C\kappa^{-n}$
(independently of the $A_i$'s). Then  the words $w_{1,n},\dots,w_{p,n}$ satisfy the property  with probability going to $1$   exponentially fast 
as $n\to+\infty$. } This follows from radial symmetry, since   $Ln\le  | w_{i,n} | \le C e^{\theta n}$ holds generically by positive drift and subexponential growth.

\subsubsection{The central tree property (see for instance \cite{BMNVW})} \label{sctp}

Let $ {w_{1,n},\dots,w_{p,n}}$ be as in the proposition. We fix $n$, and we let $\Theta$ 
 be the Stallings graph of $H$. 
 The  
  elements  $w_{i,n}$, indeed all    elements of $H$,  are represented by immersed   paths 
with both endpoints the base vertex $1$.

 The central tree property  
says that,  generically, the graph $\Theta$ looks like a rose.  The \emph{$m$-prefix} of a word is its initial subword of length $m$.

\begin{lem} \label{ctp}
 Fix $\lambda<\frac L2$. With probability going to $1$   exponentially fast 
as $n\to+\infty$, the $2k$   elements $w_{i,n}^{\pm1}$ have length $\ge \lambda n $ and have distinct 
  $\lambda 
 n$-prefixes.  
\end{lem}

We shall consistently neglect the fact that numbers such that $\lambda n$ are not necessarily integers (so that we should write $[\lambda n]$ instead).

\begin{dfn}[Central tree, outer loops]

Viewing the words $w_{i,n}^{\pm1}$ as loops based at $1$   in $\Theta$, their  initial segments of length $\lambda 
 n$     are all distinct, so form a \emph{central tree}   $\calc\subset \Theta$
with $2p$ or $2p+1$ leaves ($1$ may be a leaf). The complement of $\calc$ in $\Theta$   consists of $p$ arcs of length $>(L-2\lambda)n$ called the \emph{outer loops}.
\end{dfn}

\begin{proof}[Proof of Lemma \ref{ctp}]

Fix $n$ and $A\ge \lambda n$ (we do not use subexponential growth in this proof). 
The number of reduced words of length $A$ is $\gamma_A=2k(2k-1)^{A-1} $. 
Among those, the number of words $w$  such that $w $ and $w \m$ have the same  $\lambda 
 n$-prefix is at most 
 $\tilde \gamma_{A;n}=2k(2k-1)^{A-1-\lambda n} $, since $w$ is completely determined by its $(A-\lambda n$)-prefix. 
The probability that an element $w$ chosen at random among elements of length $A$ has the same $\lambda n$-prefix as $w\m$
is therefore bounded by $\tilde \gamma_{A;n}/ \gamma_A= (2k-1)^{ -\lambda n} $. 

 Since $(2k-1)^{ -\lambda n} $ goes to 0 exponentially fast as $n\to\infty$, 
 the Fubini-type argument  mentioned above  implies that,  for each $i$,   the $\lambda 
 n$-prefixes of $w_{i,n} $ and $w_{i,n}\m$ are generically different.

 The argument for $w_{i,n}$ and  $w_{j,n}^{\pm1}$ is similar. Fix $A_1,\dots,A_p$ bigger than $\lambda n$. The number  $\gamma_{A_1,\dots,A_p;n}$ of $p$-tuples $(w_1,\dots,w_p)$ with $ | w_i | =A_i$ is $(2k)^p\prod _{\alpha=1}^p(2k-1)^{A_\alpha-1} $. The number   $\tilde \gamma_{A_1,\dots,A_p;i,j,n}$ of those for which   $w_{i }$ and   $w_{j}^{\pm1}$ have the same $\lambda 
 n$-prefix is bounded by 
 twice  the same product, but with the term $(2k-1)^{A_j-1} $ replaced by $(2k-1)^{A_j-1-\lambda n} $, so the ratio $\tilde \gamma_{A_1,\dots,A_p;i,j,n}/\gamma_{A_1,\dots,A_p;n}$ is bounded by $ 2 (2k-1)^{ -\lambda n} $. 
 \end{proof}

\subsubsection{Whitehead minimality}

As in \cite{KSS_generic}, we use 
Whitehead's peak reduction. We refer to \cite[Section I.4]{LyndonSchupp} for the basic definitions and results
of this theory (the reader unfamiliar with it may skip the definitions and simply combine Lemma \ref{lem_relab} and Proposition \ref{lem_PR}).

It is now more convenient to work with cyclically reduced elements. 
 If $g=s_1\dots s_l$ is a cyclically  reduced word with $s_i\in X^{\pm 1}$, its cyclic permutations are the words $s_i\dots s_ls_1\dots s_{i-1}$. 
As in \cite {LyndonSchupp}, the set of all cyclic permutations of $g$ is called a cyclic word, it corresponds to a conjugacy class in $\F_k$.

 A \emph{relabeling automorphism} of $\F_k$ is an automorphism preserving $X^{\pm1}$. 

\begin{dfn}[Strictly Whitehead minimal,   {\cite[Def.~1.3]{KSS_generic}}]
  A cyclically reduced element $g\in \F_k$ is \emph{strictly Whitehead minimal} if   $|\phi(g)|>|g|$ for any Whitehead automorphism $\phi$ which is not inner and is not 
 a relabeling automorphism.  
\end{dfn}

  A cyclically reduced word $g\in \F_k$ is strictly Whitehead minimal if and only if all its cyclic conjugates are.
We thus say that the corresponding cyclic word is strictly Whitehead minimal.

We use peak reduction in the following form.

\begin{lem}[see \cite{KSS_generic}, Proposition 4.3]\label{lem_relab}

  Strictly Whitehead minimal elements have minimal length in their $\Aut(\F_k)$-orbit. 
  If two cyclically reduced elements $g,h\in \F_k$ are strictly Whitehead minimal and $\alpha(g)=h$ for some $\alpha\in\Aut(\F_k)$, then
$\alpha$ is the composition of an inner automorphism and   a relabeling automorphism;   one passes from $g$ to $h$ by  a cyclic permutation and a relabeling. \qed
\end{lem}

\subsubsection{Equidistribution}
As observed in \cite{KSS_generic}, one deduces from Proposition I.4.16 of  \cite {LyndonSchupp} that a word $g$  is  strictly Whitehead minimal if all words of length 2 in $X^{\pm1}$ appear with approximately the same frequency in $g$.  

Given a freely reduced word $g=s_1\dots s_\ell$ with $s_i\in X^{\pm 1}$,  
and  a letter $u\in X^{\pm 1}$,   let $$P_u(g)=\frac1\ell \#\{i\leq \ell\mid s_i=u\}$$ be the proportion of $u$'s among the letters of $g$.

Given a couple of letters $(u,v)\in X^{\pm1}\times X^{\pm1}$ with $u\neq v\m$,
we also define $$P_{uv}(g)=\frac1{\ell-1}\#\{i\leq \ell-1\mid s_i=u,s_{i+1}=v\},$$ the frequency of $uv$ in $g$. 

We define $P_u(g)$ and $P_{uv}(g)$ similarly if $g$ is a  
cyclic word, except that we agree that $s_{\ell+1}=s_1$ and we define $P_{uv}(g)=\frac1{\ell}\#\{i\leq \ell\mid s_i=u,s_{i+1}=v\}$.

\begin{dfn}[$\varepsilon$-equidistributed]\label{dfn_pseudorandom}
  Given $\eps>0$, say that a reduced word $g\in \F_k$ (or a  
  cyclic word representing a conjugacy class in $\F_k$) 
is    \emph{$\eps$-equidistributed} if :  
  \begin{enumerate}
  \item $|P_u(g)-\frac{1}{2k}|\leq \eps$ for every $u\in X^{\pm1}$;
  \item  
$|P_{uv}(g)-\frac{1}{2k(2k-1)} | \leq \eps$     for every couple of letters $(u,v)\in X^{\pm1}\times X^{\pm1}$
    with $u\neq v\m$
  \end{enumerate}
  (note that $2$ implies 1, with a different $\varepsilon$).
\end{dfn}

\begin{lem}[{\cite[Lemma 4.8]{KSS_generic}}]\label{lem_Whmin}
Given $k$, there exists $\eps_0$ such that, if a  cyclic word $g$  is $\eps_0$-equidistributed, then $g$ is strictly Whitehead minimal. \qed
\end{lem}

\begin{lem}[{\cite[Proposition 5.3]{KSS_generic}}]\label{sphere}
Let $\gamma_n$ be the number of reduced words of length $n$ in $\F_k$, and let $\gamma_n(\varepsilon)$ be the number of  $\eps $-equidistributed reduced words of length $n$. 
For any $\varepsilon>0$, the ratio  $\gamma_n(\varepsilon)/ \gamma_n$ goes to $1$ exponentially fast as $n\to\infty$.
 \qed
\end{lem}

We can now state:

\begin{prop}\label{lem_PR}
 Let $H$ be as in Proposition \ref{simple}.  The following holds with probability going to $1$   exponentially fast 
as $n\to+\infty$:    for every non-trivial element $g\in H $, the cyclic reduction $\bar g$ of $g$  is strictly Whitehead minimal.  
\end{prop}

\begin{proof}
We deduce this  from the preceding lemmas and the central tree property (Lemma \ref{ctp}). 
We show that $\bar g$ is $\eps_0$-equidistributed, with $\varepsilon_0$ provided by Lemma \ref{lem_Whmin}.
Fix $\varepsilon_2<\varepsilon_1<\varepsilon_0$.

 By Lemma \ref{sphere} and radial symmetry, the words $w_{i,n}$ are $\eps _2$-equidistributed generically. One obtains their cyclic reduction $\bar w_{i,n}$ by removing initial and terminal subwords, whose length is bounded by the central tree property;  applying Lemma \ref{ctp} with $\lambda$ small enough (depending on $\varepsilon_2$ and $\varepsilon_1$), we deduce that, generically,  the cyclic words $\bar w_{i,n}$ are $\eps _1$-equidistributed. 

We now consider the cyclic reduction $\bar g$ of a non-trivial   $g\in H $. It is represented by an immersed loop in the Stallings graph $\Theta$. Generically, this loop consists of arcs  of length $<2\lambda n$ contained in the central tree $\calc$ (see Subsection \ref{sctp}) and outer loops of length $>(L-2\lambda)n$.  Frequencies are controlled in outer loops, and  $\bar g$ is $\eps_0$-equidistributed if $\lambda$ is small enough (depending now on $\varepsilon_1$ and $\varepsilon_0$).
\end{proof}

\subsubsection{Matching subwords}

The following lemma is a variation on a standard fact (see Lemmas 4.5 and 4.6 of  \cite{BMNVW} and the references given there).

\begin{lem}\label{stand}
 Let $0<\beta<L/2$. With probability going to 1 exponentially fast as $n\to\infty$, 
 the $2p$ words $w_{i,n}^{\pm1}$ have length at least $ \beta n $, and   all their subwords of length $ \beta n $  are distinct. 

\end{lem}

\begin{proof}
This is similar to the proof of Lemma \ref{ctp}, but we have to use subexponential growth: it implies that, generically, $ | w_{i,n} | \le (2k-1)^{\beta n/3}$ for $i=1,\dots,p$.

    We fix $n$, and numbers $A_1,\dots,A_p$    with $\beta n\leq A_i\leq (2k-1)^{\beta n/3}$. 
  We consider words $w_i$ with  $ | w_i | =A_i$ chosen independently uniformly on the respective spheres. Thanks to the Fubini-type argument, it suffices to bound the probability that some word of length $\beta n$ appears twice in the words $w_{i }^{\pm1}$ by some $C\kappa^{-n}$, with $C,\kappa$ independent of $n,A_1,\dots,A_p$.

 First suppose that some word $u$ of length $ \beta n $ appears in both $w_i$ and $w_j^{\pm1}$ for some fixed $i,j$ with  $i\ne j$. There are $ A_j- \beta n $ possibilities for the location of $u$ within $w_j^{\pm1}$. Once this is fixed, $w_j$ is determined by the letters outside of $u$, hence by two reduced words whose lengths add up to $A_j-\beta n$.
 It follows that  the number of possibilities for $w_j$ is bounded by $\tilde\gamma_{A_j;n}=(A_j- \beta n )(2k)^2(2k-1)^{A_j- \beta n -2}$. The probability that the words $w_i$ and $w_j^{\pm1}$ have a common subword of length $ \beta n $ is thus bounded by 
  $\tilde\gamma_{A_j;i,j,n}/2k(2k-1)^{ A_j}=2k(A_j- \beta n ) (2k-1)^{ - \beta n -2}$, hence by $2k(2k-1)^{ - 2\beta n /3}$ 
 because $A_j\le(2k-1)^{\beta n/3}$.

 Now suppose that $u$ appears twice  
  in $\{w_i,w_i\m\}$.  We then have two subwords $u_1$ and $u_2$ in $w_i$, each equal to    $u^{\pm1}$.
 There are $ 4(A_i-\beta n)^2$ possibilities for the location and sign of $u_1,u_2$. Fix one.  
 
 The key remark   is the following. If we consider the set  $Z$ of cardinality $ A_i$ whose elements are the letters of $w_i$, and   an equivalence relation on $Z$ identifying each letter of $u_1$ with the corresponding letter of $u_2$, there is a subset $Y\inc Z$ of cardinality $ A_i-\beta n$, consisting of one or two intervals and  meeting each equivalence class. In particular, $w_i$ is determined by $ A_i-\beta n$ letters. We conclude by checking that the ratio between $4(A_i-\beta n)^2(2k)^2(2k-1)^{A_i-\beta n-2}$ and  $2k(2k-1)^{A_i-1}$ is bounded by some $C(2k-1)^{-\beta n/3}$   provided that $A_i\le (2k-1)^{\beta n/3}$.
\end{proof}

We generalize Lemma \ref{stand} as follows (compare Lemma 8.3 of \cite {KSS_generic}).

\begin{lem}\label{stand2}
Let $\varphi$ be a  relabeling automorphism other than  the identity. 
 Let $0<\beta<L/2$. With probability going to 1 exponentially fast as $n\to\infty$, 
 the   words $w_{i,n}^{\pm1}$ cannot contain both a word $u$ of length $\beta n$ and its image by $\varphi$.
\end{lem}

\begin{proof}
 This is proved as the previous lemma if $\varphi(u)\ne u$. If $\varphi(u)=u$, some $w_{i,n}$ has a subword of length $\beta n$ all of whose letters are fixed by $\varphi$. Since $\varphi$ is not the identity, these letters belong to a set of cardinality at most $2k-2$. It is easily checked that this happens with probability going to 0 exponentially fast.
\end{proof}

\subsubsection{$\Aut$-malnormality}

We  can now prove Proposition \ref{simple}. With probability going to 1 exponentially fast, $H$ satisfies the conclusion of 
Lemmas \ref{ctp}, \ref{stand},    \ref{stand2} and Proposition \ref{lem_PR}  (with numbers $\lambda$ and $\beta$  which we choose  so that $3\lambda<\beta<  L/9$).

Consider $\alpha\in \Aut(\F_k)$ such that $\alpha (H)\cap H\ne\{1\}$. Fix 
  a non-trivial element $h_1\in H$ such that $h_2 =\alpha(h_1)\in H$. Denote by  $\bar h_i=a_ih_ia_i\m$   the cyclic reduction of $h_i$,
and consider the automorphism $\theta=\ad_{a_2} \circ\alpha\circ \ad_{a_1}\m$ sending $\bar h_1$ to $\bar h_2$.

    The elements $  \bar h_i$ are strictly Whitehead minimal by Proposition \ref{lem_PR}. By Lemma \ref{lem_relab}, they differ by a cyclic permutation and a relabeling $\varphi$, and $\theta=ad_g\circ\varphi$ for some $g\in G$. We claim that $\varphi$ has to be the identity, so that $\theta$ and $\alpha$ are inner.
    
    We view each $\bar h_i$ as an immersed loop in the Cayley graph of $H$. By the central tree property (Lemma \ref{ctp}), they consist of short arcs contained in the central tree and outer loops. Choose a subword $u$ of $\bar h_1$ of length $3\beta n$ contained in an outer loop, hence in some $w_{i,n}$ (this is possible because $\beta<L/9$). 
The word $\varphi(u)$ appears as a subword of $\bar h_2$.   Since $\lambda<\beta/3$, some subword of length $ \beta n$ of $\varphi(u)$ is contained in an outer loop. 
  Lemma \ref{stand2} now implies that $\varphi$ is trivial.
  
  We have proved that any $\alpha$ such that   $\alpha (H)\cap H\ne\{1\}$ is inner. We deduce $\alpha\in\Inn_H(G)$ from malnormality of $H$. Malnormality of subgroups generated by elements chosen randomly in balls is known (\cite{Jitsukawa_malnormal}, \cite[Theorem 4.3]{BMNVW}). Malnormality for $H$ as in   Proposition \ref{simple} follows from Lemma \ref{stand} as in the proof of Theorem 4.3 of \cite {BMNVW}.

\bibliographystyle{alpha}
\bibliography{published,unpublished}

\newcommand{\etalchar}[1]{$^{#1}$}
\begin{thebibliography}{BMN{\etalchar{+}}13}

\bibitem[AMS16]{AMS_commensurating}
Yago Antol\'{\i}n, Ashot Minasyan, and Alessandro Sisto.
\newblock Commensurating endomorphisms of acylindrically hyperbolic groups and
  applications.
\newblock {\em Groups Geom. Dyn.}, 10(4):1149--1210, 2016.

\bibitem[BF95]{BF_stable}
Mladen Bestvina and Mark Feighn.
\newblock Stable actions of groups on real trees.
\newblock {\em Invent. Math.}, 121(2):287--321, 1995.

\bibitem[BH92]{BH_tt}
Mladen Bestvina and Michael Handel.
\newblock Train tracks and automorphisms of free groups.
\newblock {\em Ann. of Math. (2)}, 135(1):1--51, 1992.

\bibitem[BMN{\etalchar{+}}13]{BMNVW}
Fr\'{e}d\'{e}rique Bassino, Armando Martino, Cyril Nicaud, Enric Ventura, and
  Pascal Weil.
\newblock Statistical properties of subgroups of free groups.
\newblock {\em Random Structures Algorithms}, 42(3):349--373, 2013.

\bibitem[CM15]{CaMa_virtual}
Christopher~H. Cashen and Jason~F. Manning.
\newblock Virtual geometricity is rare.
\newblock {\em LMS J. Comput. Math.}, 18(1):444--455, 2015.

\bibitem[DGO17]{DGO_HE}
F.~Dahmani, V.~Guirardel, and D.~Osin.
\newblock Hyperbolically embedded subgroups and rotating families in groups
  acting on hyperbolic spaces.
\newblock {\em Mem. Amer. Math. Soc.}, 245(1156):v+152, 2017.

\bibitem[GL15]{GL6}
Vincent Guirardel and Gilbert Levitt.
\newblock Splittings and automorphisms of relatively hyperbolic groups.
\newblock {\em Groups Geom. Dyn.}, 9(2):599--663, 2015.

\bibitem[GL17]{GL_JSJ}
Vincent Guirardel and Gilbert Levitt.
\newblock J{SJ} decompositions of groups.
\newblock {\em Ast\'{e}risque}, (395):vii+165, 2017.

\bibitem[Jit02]{Jitsukawa_malnormal}
Toshiaki Jitsukawa.
\newblock Malnormal subgroups of free groups.
\newblock In {\em Computational and statistical group theory ({L}as {V}egas,
  {NV}/{H}oboken, {NJ}, 2001)}, volume 298 of {\em Contemp. Math.}, pages
  83--95. Amer. Math. Soc., Providence, RI, 2002.

\bibitem[KSS06]{KSS_generic}
Ilya Kapovich, Paul Schupp, and Vladimir Shpilrain.
\newblock Generic properties of {W}hitehead's algorithm and isomorphism
  rigidity of random one-relator groups.
\newblock {\em Pacific J. Math.}, 223(1):113--140, 2006.

\bibitem[LS01]{LyndonSchupp}
Roger~C. Lyndon and Paul~E. Schupp.
\newblock {\em Combinatorial group theory}.
\newblock Classics in Mathematics. Springer-Verlag, Berlin, 2001.
\newblock Reprint of the 1977 edition.

\bibitem[MS17]{MaSi_random}
Joseph Maher and Alessandro Sisto.
\newblock Random subgroups of acylindrically hyperbolic groups and hyperbolic
  embeddings, 2017.
\newblock arXiv:1701.00253.

\bibitem[Ota92]{Otal_certaines}
Jean-Pierre Otal.
\newblock Certaines relations d'\'{e}quivalence sur l'ensemble des bouts d'un
  groupe libre.
\newblock {\em J. London Math. Soc. (2)}, 46(1):123--139, 1992.

\bibitem[Pau89]{Pau_Gromov}
Fr{\'e}d{\'e}ric Paulin.
\newblock The {G}romov topology on {$\mathbb {R}$}-trees.
\newblock {\em Topology Appl.}, 32(3):197--221, 1989.

\bibitem[Pau97]{Pau_automorphismes}
Fr{\'e}d{\'e}ric Paulin.
\newblock Sur les automorphismes ext\'erieurs des groupes hyperboliques.
\newblock {\em Ann. Sci. \'Ecole Norm. Sup. (4)}, 30(2):147--167, 1997.

\bibitem[Sch87]{Schupp_inner}
Paul~E. Schupp.
\newblock A characterization of inner automorphisms.
\newblock {\em Proc. Amer. Math. Soc.}, 101(2):226--228, 1987.

\end{thebibliography}

\begin{flushleft}
Vincent Guirardel\\
Univ Rennes, CNRS, IRMAR - UMR 6625, F-35000 Rennes, France\\
\emph{e-mail: }\texttt{vincent.guirardel@univ-rennes1.fr}\\[8mm]

Gilbert Levitt\\
Laboratoire de Math\'ematiques Nicolas Oresme (LMNO)\\
Universit\'e de Caen et CNRS (UMR 6139)\\
(Pour Shanghai : Normandie Univ, UNICAEN, CNRS, LMNO, 14000 Caen, France)\\
 \emph{e-mail: }\texttt{levitt@unicaen.fr}\\[8mm]

\end{flushleft}

\end{document}